\documentclass[amstex,10pt,reqn]{amsart}
\usepackage{fancyhdr}
\usepackage{geometry}
\usepackage{amsmath,amsfonts,amsthm,enumerate,float,mathtools,url,enumitem}
\usepackage[numbers]{natbib}
\newtheorem{lemma}{Lemma}[section]

\newtheorem{thm}{Theorem}[section]
\newtheorem{corollary}{Corollary}[section]

\theoremstyle{remark}
\newtheorem{remark}{Remark}[section]
\numberwithin{equation}{section}
\usepackage[utf8]{inputenc}
\usepackage{titlesec}
\titleformat{\section}[hang] 
  {\normalfont\Large\bfseries}{\thesection}{1em}{} 

\usepackage[T2A]{fontenc}    
\usepackage[russian, english]{babel}
\usepackage{parskip}
\usepackage{amssymb}
\usepackage{microtype}
\begin{document}
\tolerance=1000  % Increase tolerance for overfull hboxes
\sloppy  % Allow more flexible line breaking
\title[]{Modified Kantorovich-type Sampling Series in Orlicz Space Frameworks.}
\maketitle
\begin{center}
\text{Pooja Gupta} \vskip0.2in
Department of Mathematics,\\ J. C. Bose University of Science
and Technology, YMCA,\\ Faridabad-121 006, India\\[0pt]
\textbf{E-mail:} \url{poojaguptaiitr@gmail.com}.
\end{center}
\begin{abstract}
This study examines a modified Kantorovich approach applied to generalized sampling series. The paper establishes that the approximation order to a function using these modified operators is atleast as good as that achieved by classical methods by using some graphs. The analysis focuses on these series within the context of Orlicz space \( L^{\eta}(\mathbb{R}) \), specifically looking at irregularly spaced samples. This is crucial for real-world applications, especially in fields like signal processing and computational mathematics, where samples are often not uniformly spaced. The paper also establishes a result on modular convergence for functions \( g \in L^{\eta}(\mathbb{R}) \), which includes specific cases like convergence in \( L^{p}(\mathbb{R}) \)-spaces, \( L \log L \)-spaces, and exponential spaces. The study then explores practical applications of the modified sampling series, notably for discontinuous functions and provides graphs to illustrate the results.
\vspace{1em}

\vspace{0.5cm}
\noindent\textbf{2000 A M S Mathematics Subject classification}: 41A25, 47A58, 41A35, 46E30, 94A12, 47B38.
\end{abstract}
\maketitle
\section{Introduction}
The development of approximation theory saw a significant milestone with the introduction of Bernstein polynomials in 1912. Using these polynomials, a constructive approach was shown for proving the Weierstrass Approximation Theorem, which asserts that any continuous function on a closed interval can be uniformly approximated by polynomials. Bernstein's formulation, which is defined as follows:
\begin{eqnarray*}
(\tau_m g)(y) &=& \sum_{i=0}^{m} g\left(\frac{i}{m}\right) \vartheta_{i,m}(y),\,\,\,\,\,\,\, \vartheta_{i,m}(y) = \binom{m}{i} y^i (1 - y)^{m - i},
\end{eqnarray*}
offered a tangible approach to approximation and inspired further exploration into approximating operators. S.N. Bernstein \cite{BT} showed that 
\begin{eqnarray}
    \displaystyle \lim_{n\rightarrow{\infty}}||\tau_ng-g||_{C[0,1]}:=\displaystyle \lim_{n\rightarrow{\infty}} \sup_{y\in [0,1]}|\tau_ng(y)-g(y)|=0.
\end{eqnarray}
Following this, the focus shifted to optimal approximations, with Chebyshev polynomials playing a pivotal role in minimizing the maximum error of approximation. These developments laid the groundwork for broader research, including generalizations like Kantorovich operators and applications in function spaces such as $L^p$ and Orlicz spaces. Kantorovich operators are a class of mathematical operators used in approximation theory. They generalize Bernstein polynomials by incorporating integration, making them suitable for approximating functions that are not necessarily continuous but belong to broader function spaces, such as $L^p$ spaces. These are typically defined for functions \( f \) belonging to \( C[0,1] \) and \( L^p(0,1) \), where \( x \) lies in the interval \( 0 \leq x \leq 1 \), as follows:
\begin{eqnarray*}
    (\upsilon_mg)(y):=\displaystyle\sum_{i=0}^m \bigg\{(m+1)\int_{\frac{i}{(m+1)}}^{\frac{i+1}{(m+1)}} g(u) du\bigg\} \vartheta_{i,m}(y),
\end{eqnarray*}
These polynomials $\upsilon_m(f)$ were introduced by L.V. Kantorovich in 1930 \cite{LV}, and G.G. Lorentz later demonstrated their properties in 1937 \cite{LR}. Kantorovich proved that 
\begin{eqnarray*}
    \displaystyle \lim_{m\rightarrow{\infty}}||\upsilon_mg-g||_{L^p(0,1)}:=\displaystyle\lim_{m\rightarrow{\infty}} 
\left( \int_{0} ^{1} \big| \upsilon_m g(y) - g(y) \big|^p \, dy \right)^{1/p} = 0.
\end{eqnarray*}
\hspace{2em} Now, the Weierstrass approximation theorem establishes that any continuous function on a closed and bounded interval can be approximated in a uniform manner by polynomials. However, many real-world problems require approximation on the entire real line e.g., in signal processing, where functions often have infinite support. Generalized sampling series provides a framework for such approximations on the entire real line $\mathbb{R}$. These series are defined for $f\in C(\mathbb{R})$ as follows:
\begin{eqnarray}
    (T_w^{\chi}g)(x):=\displaystyle\sum_{i=-\infty}^{i=\infty}g\left(\frac{i}{w}\right)\chi(wx-i)\,\,\,\,\,\,(x\in \mathbb{R})
\end{eqnarray}
provided a suitable kernel function , $\chi \in L^{1}(\mathbb{R})\cap C(\mathbb{R})$ is used. This series approximates $f$ by sampling it at equally spaced points $i/w$ on the real line and combining these values with the weights provided by the kernel $\chi.$ It generalizes classical sampling theory to more flexible kernels and can handle irregular or non-bandlimited signals, depending on the choice of the kernel. 
Generalized sampling theory extends classical approaches by utilizing more adaptable kernels, enabling the approximation of irregular and non-bandlimited signals. The original work on generalized sampling series, focusing on a specific kernel, was presented by M. Thesis \cite{MTH} in 1919. However, the broader exploration of these series for arbitrary kernel functions began in Aachen in 1977 and has since been extensively developed, with significant contributions in various fields such as signal processing, approximation theory and numerical analysis. The flexibility of choosing different kernels in this framework has allowed for advancements in handling more complex, real-world data, where traditional, uniform sampling methods fall short.
\par
\hspace{2em}Here, the operators $T_w^{\chi}$ are bounded and linear, mapping $C(\mathbb{R})$ into itself. These operators possess a finite operator norm $||T_w^{\chi}||_{[C(\mathbb{R}),C(\mathbb{R})]}=\sup_{u\in \mathbb{R}}\displaystyle\sum_{i\in \mathbb{Z}}|\chi(u-i)|$ and 
\begin{eqnarray*}
    \displaystyle\lim_{w\rightarrow \infty}||T_w^{\chi}g-g||_{C(\mathbb{R})}=0.
\end{eqnarray*}
These operators often assume a certain degree of smoothness or continuity in the function $g$. If the target function does not possess these properties, the approximation quality may deteriorate considerably. A particularly insightful example illustrates this limitation: Let \( g\) be a function such that \( g(m) = 2 \) for all \( m \in \mathbb{Z} \), and suppose that \( g \in L^p(\mathbb{R}) \). Then, for \( x \in \mathbb{R} \) and \( w = 1 \), we have \( (T_1^\chi g)(x) := 2 \), since \( \sum_{i=-\infty}^{\infty} \chi(x - i) = 1 \). However, \( (T_1^\chi g)(x) \) does not belong to \( L^p(\mathbb{R}) \).

 \par
\hspace{2em}C. Bardaro and collaborators \cite{CB} addressed this issue by modifying existing operators and introducing a new type of generalized sampling series based on the Kantorovich framework. They replaced the values $g\left(\frac{i}{w}\right)$ by an average of $g$ namely $w\int_{i/w}^{(i+1)/w}$. This adjustment allows for greater flexibility in handling non-uniformly spaced data points. The series they developed can be expressed as follows:
 \begin{eqnarray}\label{eq1}
  (S_w^{\chi}g)(y):=\displaystyle\sum_{i=-\infty}^{i=\infty}w\left(\int_{i/w}^{i+1/w}g(u)du\right)\chi(wy-i)\,\,\,\,(y\in \mathbb{R})
  \end{eqnarray}
  for $g\in L^p(\mathbb{R}),1\leq p<\infty$ and they proved that
  \begin{eqnarray*}
      \displaystyle\lim_{w\rightarrow \infty}||S_wg-g||_{L^p(\mathbb{R})}
      = \displaystyle\lim_{w\rightarrow \infty}||S_wg-g||_p=0.
  \end{eqnarray*}
We can see that the operator $(\ref{eq1})$ involved integration over the interval $[0,1]$, with \( g \) evaluated at \( \frac{i + t}{n+1}, \,\,0\leq t\leq 1 \). To standardize the operator and preserve linearity in \( g \), we modify the integration to a fixed interval \( [0, 1] \), with the argument of \( g \) transformed nonlinearly by \( \frac{i + t^\alpha}{n+1} \) and define a more general family of operators, $T_w$. This adjustment captures more complex features of \( g \) while maintaining linearity. The nonlinear term \( t^{\alpha}\) adapts to local variations, making it well-suited for signals with rapid changes.
\par
\hspace{2em}A key result presented in Corollary $\ref{col2}$ of this paper asserts that for each $g\in L^p(\mathbb{R})$ with $1\leq p<\infty,$ the following holds:
\begin{eqnarray*}
    \displaystyle\lim_{w\rightarrow \infty}||T_wg-g||_{L^p(\mathbb{R})}
      = \displaystyle\lim_{w\rightarrow \infty}||T_wg-g||_p=0.
\end{eqnarray*}
In this work, we broaden the scope of our analysis beyond classical \( L^p(\mathbb{R}) \)-spaces by considering Orlicz spaces defined by convex \( \eta \)-functions. This approach offers a more unified framework for studying the convergence properties of the operators under consideration. Orlicz spaces serve as a broader framework that extends \( L^p \)-spaces and includes a diverse range of other function spaces. Among these are the \( L^{\alpha} \log^{\beta} L \)-spaces and exponential Orlicz spaces. The \( L^{\alpha} \log^{\beta} L \)-spaces are characterized by the \( \eta \)-function \( \eta_{\alpha, \beta}(t) = t^{\alpha} \log^{\beta}(e + t) \) with \( \alpha \geq 1 \) and \( \beta > 0 \), while exponential Orlicz spaces are defined using the \( \eta_{\alpha}(t) = e^{t^\alpha} - 1 \) for \( \alpha > 0 \).
 In various branches of analysis, these functions are of great importance, especially in the study of differential equations, harmonic analysis, and the theory of operator algebras. This extended framework allows for a more comprehensive treatment of functions with varying growth behaviors and provides a flexible tool for analyzing operator convergence in diverse settings.
\par
\hspace{2em}The \( L^{\alpha} \log^{\beta} L \)-spaces are extensively explored in connection with the \emph{Hardy-Littlewood maximal function} which is a fundamental tool in harmonic analysis. These spaces, which generalize \( L^p \)-spaces by incorporating logarithmic factors, are particularly useful for analyzing functions with growth behaviors that exceed those captured by traditional \( L^p \)-norms. The properties of these spaces enable precise control over maximal inequalities, offering important insights into the behavior of averages of functions.

On the other hand, \emph{exponential Orlicz spaces} play a key role in the theory of \emph{Sobolev space embeddings} (see, e.g., \cite{SE,RLM,DE,LL}) which are central to understanding the regularity of solutions to partial differential equations (PDEs). Exponential spaces, defined by \( \eta_{\alpha}(t) = e^{t^\alpha} - 1 \), provide a natural framework for embedding Sobolev spaces into other functional spaces. These embeddings are essential for studying the regularity and growth properties of solutions to elliptic and parabolic PDEs, facilitating the analysis of their behavior in a variety of contexts.The theory has been extensively developed through the contributions of multiple researchers; for detailed discussions, refer to [9, 10, 24, 32, 33, 35]. Several authors have contributed to the development of the theory of sampling Kantorovich operators see, for example \cite{AAG,ACAR,BTP,CT,CP,CTM,CADE,KUM}.
\par
In this work, we further investigate \( L^p \)-convergence in the context of sampling sums, a topic first explored in \cite{BA}. 
\section {Preliminaries}
We define \( C(\mathbb{R}) \) as the space of uniformly bounded functions \( g: \mathbb{R} \to \mathbb{R} \), equipped with the sup-norm \( ||g||_{\infty} \). The subspace \( C_c(\mathbb{R}) \subset C(\mathbb{R}) \) includes those functions that have compact support. In addition, \( M(\mathbb{R}) \) represents the collection of all measurable functions on \( \mathbb{R} \).

Let \( \eta: \mathbb{R}_0^+ \to \mathbb{R}_0^+ \) be a convex \( \eta \)-function, for which the following condition holds:

\begin{enumerate}[label=\arabic*.]
    \item \( \eta(0) = 0 \), \( \eta(x) > 0 \) for all \( x > 0 \), and \( \eta \) is non-decreasing on \( \mathbb{R}_0^+ \).
\end{enumerate}

Next, we define the functional
\[
I^{\eta}[h] := \int_{\mathbb{R}} \eta(|h(y)|) \, dy \quad \text{for} \quad h \in M(\mathbb{R}).
\]
As established in the literature (see, for example, \cite{MLE,BDL}), the functional \( I^{\eta} \) is convex and modular on \( M(\mathbb{R}) \). The Orlicz space related to \( \eta \) is defined as:

\[
L^{\eta}(\mathbb{R}) = \{ h \in M(\mathbb{R}) : I^{\eta}[\lambda h] < \infty \text{ for some } \lambda > 0 \}.
\]

The space \( L^{\eta}(\mathbb{R}) \) is a vector space, with the corresponding subspace defined as:

\[
E^{\eta}(\mathbb{R}) = \{ h \in M(\mathbb{R}) : I^{\eta}[\lambda h] < \infty \text{ for all } \lambda > 0 \}.
\]

This is referred as the space of finite elements within \( L^{\eta}(\mathbb{R}) \). It is evident that \( C_c(\mathbb{R}) \) is a subset of \( E^{\eta}(\mathbb{R}) \).

\par
Generally, \( E^\eta(\mathbb{R}) \) is a strict subspace of \( L^\eta(\mathbb{R}) \), but the two spaces are identical if and only if the function \( \eta \) fulfills the \( \Delta_2 \)-condition. This condition ensures that we can find a constant \( M > 0 \) such that for each \( u \in \mathbb{R}_0^+ \),

\begin{eqnarray}\label{cond1}
    \eta(2u) \leq M \eta(u),
\end{eqnarray}
which guarantees that the growth of \( \eta \) does not exceed a certain bound when its argument is doubled. The \( \Delta_2 \)-condition plays a crucial role in ensuring certain regularity properties of the corresponding Orlicz spaces, particularly when it comes to the equivalence of different modes of convergence. Examples of functions \( \eta \) that satisfy this condition include the power function \( \eta(u) = u^p \) for \( 1 \leq p < \infty \), as well as the more general form \( \eta(u) = u^\alpha \log^\beta(e + u) \), where \( \alpha \geq 1 \) and \( \beta > 0 \). These functions are widely studied in the theory of Orlicz spaces, and the corresponding spaces \( E^\eta(\mathbb{R}) \) thus include important classes such as the Lebesgue spaces \( L^p(\mathbb{R}) \) and the \( L \log L \)-spaces, which are essential in functional analysis and related fields.

On the other hand, there are classes of functions, such as \( \eta(t) = e^{t^\alpha} - 1 \) for \( \alpha > 0 \), that define exponential spaces where \( E^\eta(\mathbb{R}) \) is a strict subset of \( L^\eta(\mathbb{R}) \). These exponential Orlicz spaces display distinct structural features when compared to those generated by polynomial functions. The inclusion \( E^\eta(\mathbb{R}) \subsetneq L^\eta(\mathbb{R}) \) indicates that the convergence based on the modular in \( E^\eta(\mathbb{R}) \)) is strictly weaker than that in \( L^\eta(\mathbb{R}) \), i.e. some functions may belong to \( L^\eta(\mathbb{R}) \) without being elements of \( E^\eta(\mathbb{R}) \).

In \( L^\eta(\mathbb{R}) \)-spaces, modular convergence plays a crucial role. A net of functions \( (h_w)_{w > 0} \subset L^\eta(\mathbb{R}) \) is said to converge modularly to a function \( h \in L^\eta(\mathbb{R}) \) if
\[
\lim_{w \to \infty} I^\eta[\lambda(h_w - h)] = 0
\]
for a certain \( \lambda > 0 \), where \( I^\eta \) denotes the modular functional associated with \( \eta \). This convergence captures a kind of weak convergence in the space, based on the behavior of the modular functional rather than direct pointwise or norm-based criteria.

A stronger form of convergence, norm convergence, can also be introduced. Specifically, the norm \( \| h \|_\eta \) is defined as
\[
\| h \|_\eta := \inf \left\{ \lambda > 0 : I^\eta\left(\frac{h}{\lambda}\right) \leq 1 \right\},
\]
which imposes a stricter requirement on the function than modular convergence. It is well known that norm convergence in \( L^\eta(\mathbb{R}) \) implies modular convergence and the equivalence of the two notions holds if and only if \( \eta \) satisfies the \( \Delta_2 \)-condition.
This equivalence is a key property in the theory of Orlicz spaces, ensuring that both types of convergence lead to the same topological structure in spaces where \( \eta \) is \( \Delta_2 \)-regular. 

Moreover, norm convergence and modular convergence can differ in the absence of the \( \Delta_2 \)-condition. For instance, in exponential Orlicz spaces, such as those generated by \( \eta(t) = e^{t^\alpha} - 1 \) with \( \alpha > 0 \), norm convergence is strictly stronger than modular convergence. This distinction arises due to the rapid growth of \( \eta(t) \) and the fact that the modular-based convergence in such spaces does not capture the full behavior of functions that are not norm-convergent, highlighting the intricate relationship between the modular and norm structures in different Orlicz spaces.
\section{Advanced Generalization of Kantorovich’s Sampling Series}
This section presents a set of discrete operators that characterize the Kantorovich framework for generalized sampling operators, to be analyzed in the following sections.

\par 
A function \( \chi: \mathbb{R} \to \mathbb{R} \) is classified as a \textit{kernel} when it satisfies the following properties:

\begin{enumerate}[label=\arabic*.]
    \item  \( \chi \in L^1(\mathbb{R}) \), indicating that it is integrable over \( \mathbb{R} \), and \( \chi \) remains bounded within a neighborhood around the origin. Moreover, for every \( z \in \mathbb{R} \), the following summation condition is satisfied:

    \begin{eqnarray}\label{eql1}
    \sum_{i \in \mathbb{Z}} \chi(z-i) = 1.
    \end{eqnarray}
    \item A constant \( \beta > 0 \) can be found such that:
    \begin{eqnarray}\label{eql2}
        m_{\beta}(\chi) := \sup_{z \in \mathbb{R}} \sum_{i \in \mathbb{Z}} |\chi(z - i)| |z - i|^\beta < +\infty.
    \end{eqnarray}
\end{enumerate}
Given a kernel \( \chi \), we introduce a collection of operators \( (T_w)_{w > 0} \), where for each \( w \), \( T_w \) operates on a locally integrable function  \( g: \mathbb{R} \to \mathbb{R} \) as indicated below:
\begin{eqnarray}\label{eq2}
   (T_w g)(y) = \sum_{i \in \mathbb{Z}} \chi(w y - i) \left[ \int_0^1 g \left( \frac{i + t^\alpha}{w + 1} \right) \, dt \right], \quad \text{for} \,\,\, y \in \mathbb{R}. 
\end{eqnarray}
In this case, \( g \) is assumed to be locally integrable, ensuring the sum converges for every \( y \in \mathbb{R} \). The parameter \( w > 0 \) regulates the scale of the kernel and influences the operator's behavior, while \( \alpha > 0 \) serves as a constant that shapes the function within the integral.
\par
It is evident that if \( \alpha = 1 \) is chosen in equation (\ref{eq2}), then we get the Sampling Kantorovich operators as defined in $(\ref{eq1})$.
\begin{lemma}\label{l1}
   Given the previously stated assumptions about the kernel  \( \chi \), it can be concluded that
    \begin{enumerate}[label=(\arabic*)]
    \item \label{it1} $\mu_{0,\pi}(\chi):=\sup_{z \in \mathbb{R}}\displaystyle \sum_{i \in \mathbb{Z}} |\chi(z- i)|<+\infty$
    \item  \label{it2} For every value of $\gamma>0$,
    \[
    \displaystyle\lim_{w\rightarrow \infty}\displaystyle \sum_{|wy-i|>\gamma w}|\chi(w y - i)|=0,\]
   in a way that is independent of  $y\in \mathbb{R};$
    \item \label{it3} For each pair of values \( \gamma > 0 \) and \( \epsilon > 0 \), a constant \( P > 0 \) can be found that satisfies the condition
   \[\int_{|y|>P}w|\chi(wy-i)|dy<\epsilon\]
   for sufficiently large \( w > 0 \) and \( i \) such that \( \frac{i}{w} \in [-\gamma, \gamma] \).

   \end{enumerate}
\end{lemma}
   \begin{proof}
Consider the interval $[-d, d]$ where $|\chi(z)| \leq C$ for some constant $C > 0$ and for all $z \in [-d, d]$. For ease of notation, we assume that \( d \leq 1 \). We define
       \[
       \displaystyle \sum_{i \in \mathbb{Z}} |\chi(z- i)|=\left(\displaystyle \sum_{|z-i|\leq d}+\displaystyle \sum_{|z-i|>d}\right)|\chi(z- i)|=S_1+S_2
       \]
       Since $d \leq 1$, the sum in $S_1$ includes at most two terms, which gives $S_1 \leq 2C$. For $S_2$, we have
       \[
       S_2\leq \frac{1}{d^{\beta}}\displaystyle\sum_{|z-i|>d}|\chi(z -i)| |z - i|^\beta \leq \frac{1}{d^{\beta}}\mu_{\beta,\pi}(\chi).
       \]
       Therefore, \ref{it1} has been established. Regarding \ref{it2}, let assume that \( \gamma > 0 \) is fixed. For every \( w > 0 \), it is easy to show that

       \[
       \displaystyle\sum_{|wy-i|>\gamma w} |\chi(wy - i)| \leq \frac{1}{(\gamma w)^{\beta}}\mu_{\beta,\pi}(\chi)\]
       Thus, the result follows. Finally, for property \ref{it3}, let $\epsilon > 0$ be given. Given that \( \chi \) is integrable over \( \mathbb{R} \), we can conclude the existence of a constant \( \widetilde{P} > 0 \) such that

       \[
       \int_{|y|>\widetilde{P}}|\chi(y)|dy <\epsilon.
       \]
     Let \( \gamma > 0 \). For each \( w \geq 1 \), define the integers \( i \) such that \( \frac{i}{w} \in [-\gamma, \gamma] \), which is equivalent to \( i \in [-w\gamma, w\gamma] \). Suppose \( P > 0 \) is chosen sufficiently large such that \( P - \gamma > \widetilde{P} \).
 Then, we have

       \begin{eqnarray*}
           \int_{|y|>P} w|\chi(wy-i)|dx &=&\left\{\int_{-\infty}^{-Pw-i}|\chi(v)|dv+\int_{Pw-i}^{\infty}|\chi(v)|dv\right\}\\
       &\leq& \int_{|v|>(P-\gamma)w}|\chi(v)dv| \leq \int_{|v|>\widetilde{P}}\chi(v)|dv< \epsilon,
       \end{eqnarray*}
       This concludes the demonstration of the lemma.
   \end{proof}
\begin{remark}\label{re1}
    \begin{enumerate}[label=(\alph*)]
        \item \label{ita} 
        By applying Lemma \ref{l1} (1), it follows that the operators \( T_w g \) are properly defined for all functions \( g \in L^\infty(\mathbb{R}) \) and this result holds universally for these functions. In particular,
        \[
        |(T_w g)(x)| \leq \mu_{0,\pi}(\chi) ||g||_{\infty}.
        \]
        \item \label{itb} 
        The condition \((\ref{eql2})\) for the kernel functions \(\chi\) holds if \(\chi(y)\) behaves asymptotically as \(o(y^{1-\beta-\epsilon})\) as \(y \to \pm\infty\), where \( \epsilon > 0 \) is a positive constant that can be taken as small as desired.

        \item \label{itc} 
        The moment condition \((\ref{eql1})\) is equivalent to the following expression for \(\mathcal{X}(i)\):
\[
\mathcal{X}(i) :=
\begin{cases}
    1 & \text{for } i = 0, \\
    0 & \text{for } i \in \mathbb{Z} \setminus \{0\}.
\end{cases}
\]
        where \(\mathcal{X}(l)\) is defined as the Fourier transform associated with \( \chi \).
of \(\chi\); see (\cite{BTS}, Lemma 4.2), given by
\[
\mathcal{X}(m) := \int_{-\infty}^{\infty} \chi(u)e^{-ilu} \, du, \quad l \in \mathbb{R}.
\]

    \end{enumerate}
\end{remark}

\section{Convergence theorems}
\begin{thm}\label{thm1}
  Let \( g\in M(\mathbb{R}) \) be a function that is bounded on \( \mathbb{R} \). For any \( y \in \mathbb{R} \) where \( g \) is continuous, the following pointwise convergence holds:
\[
\lim_{w \to \infty} (T_w g)(y) = g(y).
\]
Moreover, if \( g \in C(\mathbb{R}) \), we have uniform convergence:
\[
\lim_{w \to \infty} \| T_w g - g \|_{\infty} = 0.
\]
\end{thm}
\begin{proof}
We will concentrate on proving the second part of the theorem, as the first part can be derived easily using similar methods. Specifically, there exists a constant \( \gamma > 0 \) such that \( |g(u) - g(v)| < \epsilon \) holds whenever \( |u - v| < \gamma \). Let \( \bar{w} \) be chosen so that for all \( w \geq \bar{w} \), \( \delta w^{-1} < \gamma / 2 \). For \( w \geq \bar{w} \), we then have:
\begin{eqnarray*}
(T_wg)(y)-g(y)|&\leq& \displaystyle\sum_{i\in \mathbb{Z}}\left|\chi(wy-i)\right|\int_0^1\left|g\left(\frac{i+t^{\alpha}}{w+1}\right)-g
(y)\right|dt\\
&=& \left(\displaystyle\sum_{|wy-i|\leq w\gamma/3}+\displaystyle\sum_{|wy-i|>w\gamma/3}\right)|\chi(wy-i)|\int_0^1\left|g\left(\frac{i+t^{\alpha}}{w+1}\right)-g
(y)\right|dt\\
&=&J_1+J_2
\end{eqnarray*}
For $t\in [0,1]$, and if $|wy-i|\leq w\gamma/3$
\begin{eqnarray*}
    \left|\left(\frac{i+t^{\alpha}}{w+1}\right)-y\right|&\leq& \left|\frac{i}{w+1}-y\right|+\left|\frac{t^{\alpha}}{w+1}\right|\\
    &\leq& \left|\frac{i-wy}{w+1}\right|+\left|\frac{y}{w+1}\right|+\left|\frac{t^{\alpha}}{w+1}\right|\\
    &\leq& \frac{w\gamma/3}{w+1}+\frac{\gamma}{3}+\frac{\gamma}{3}\\
    &\leq& 3\frac{\gamma}{3}=\gamma
\end{eqnarray*}
Thus
\begin{eqnarray*}
    J_1\leq \displaystyle\sum_{|wy-i|\leq w\gamma/3}|\chi(wy-i)|\int_0^1 \epsilon dt<\mu_{0,\pi}(\chi)\epsilon
\end{eqnarray*}
    Moreover ,
    \begin{eqnarray*}
        J_2\leq ||f||_{\infty} \displaystyle\sum_{|wy-i|>w\gamma/3}|\chi(wx-i)|,
    \end{eqnarray*}
    Hence, the conclusion follows directly from property \ref{it2} in Lemma \ref{l1}.
\end{proof}
We now proceed with the proof of the following claim.
\begin{thm}\label{thma}
    Let \( g \in C_c(\mathbb{R}) \) and \( \lambda > 0 \). Then, we have the following convergence:
\[
\lim_{w \to \infty} \| T_w g - g \|_{\eta} = 0.
\]
\end{thm}
\begin{proof}
    We now proceed to prove that for every \( \lambda > 0 \),
\[
\lim_{w \to \infty} I^{\eta}[\lambda(T_w g - g)] = 0.
\]
By applying Theorem \ref{thm1}, we conclude:
\[
\lim_{w \to \infty} \eta\left(\lambda \| T_w g - g\|_{\infty}\right) = 0,
\]
for all \( \lambda > 0 \).
Let \( \lambda > 0 \) be arbitrary, with \( \epsilon > 0 \) and \( \beta > 0 \) fixed. Suppose that the function \( g \) has its support contained within the interval \( [-\bar{\beta}, \bar{\beta}] \), and assume that \( \beta > \bar{\beta} + 1 \). Consider the interval \( [-\beta, \beta] \). Then, for \( i \) outside the range \( [-w\beta, w\beta] \), one can find a sufficiently large \( w > 0 \) such that

\[
\frac{i+t^{\alpha}}{w+1}\notin [-\bar{\beta}, \bar{\beta}]
\]
for every $t\in [0,1]$ and $\alpha>0,$ and so
\[
\int_0^1 g \left( \frac{i + t^\alpha}{w + 1} \right) \, dt=0.
\]
Now,by Lemma $\ref{l1}(\it3)$, a constant $P>0$ can be found such that 
   \[\int_{|y|>P}w|\chi(wy-i)|dx<\epsilon\]
   for $w>0$ large enough and $m$ such that $i\in [-w\beta,w\beta].$ Utilizing Jensen's inequality along with the Fubini-Tonelli theorem, we can derive the following important result:
   \begin{eqnarray*}
       &&\int_{|y|>M}\eta(\lambda|(T_wg)(y)|)dy\\
       &\leq& \displaystyle\sum_{i\in [-\beta w,\beta w]}\frac{1}{w\mu_{0,\pi}(\chi)}\eta\left[\lambda\mu_{0,\pi}(\chi)\int_0^1 \left|g \left( \frac{i + t^\alpha}{w + 1}\right| \right) \, dt\right]\int_{|y|>P}w|\chi(wy-i)|dx\\
       &\leq&\frac{\epsilon}{w\mu_{0,\pi}(\chi)}\displaystyle\sum_{i\in [-\beta w,\beta w]}\eta\left(\lambda\mu_{0,\pi}(\chi)||g||_{\infty}\right).
   \end{eqnarray*}
   As a result, the total number of terms in the series associated with the last term of the inequality is restricted by \( 2([\beta w] + 1) \), ensuring a controlled growth rate of the summation. As a result, 
   \[
   \int_{|y|>M}\eta(\lambda|(T_wg)(y)|)dy\leq \frac{2\epsilon}{\mu_{0,\pi}(\chi)}2([\beta w]+1)\eta\left(\lambda\mu_{0,\pi}(\chi)||g||_{\infty}\right).
   \]
   Consider a measurable set \( D \subset \mathbb{R} \) such that \( |D| < \frac{\epsilon}{\mu_{0,\pi}(\chi) \eta(\lambda \|g\|_{\infty})} \). Then, as indicated in Remark \(\ref{re1} \, (\it a)\),
 \[
  \int_D \eta(\lambda|(T_wg)(y)|)dy\leq \int_D \eta\left(\lambda\mu_{0,\pi}(\chi)||g||_{\infty}\right)dy<\epsilon
  \]
 Thus, the integrals
\[
\int_{(.)} \eta\left( \lambda \left| (T_w g)(y) - g(y) \right| \right) \, dx
\]
are equi-absolutely continuous, which permits the application of the Vitali convergence theorem, leading to the desired conclusion.
\end{proof}
\begin{thm}\label{thm2}
    Given any \( g \in L^{\eta}(\mathbb{R}) \), the following holds true:

    \begin{eqnarray}\label{eq5}
        I^{\eta}[\lambda T_wg]\leq \frac{||\chi||_1}{\mu_{0,\pi}(\chi)}I^{\eta}[\lambda \mu_{0,\pi}(\chi)g]\,\,\,\,\, (\lambda>0).
    \end{eqnarray}
    Specifically, \( T_w g \) is well-defined, and for any \( g \in L^{\eta}(\mathbb{R}) \), it follows that \( T_w g \in L^{\eta}(\mathbb{R}) \).
\end{thm}
\begin{proof}
  We will consider the case where the expression \( I^{\eta} \) on the right-hand side of equation (\ref{eq5}) is finite and for convenience, we set \( \lambda = 1 \). By applying Jensen's inequality twice, we derive the following inequality:
    \begin{eqnarray*}
       &&\int_{\mathbb{R}} \eta(|(T_wg)(y)|)dy\\
        &\leq& \frac{1}{\mu_{0,\pi}(\chi)}\displaystyle\sum_{i\in \mathbb{Z}}\left\{\int_0^1 \eta\left(\mu_{0,\pi}(\chi)\left|g \left( \frac{i+ t^\alpha}{w + 1}\right)\right| \right)dt.\int_{\mathbb{R}}|\chi(wy-i)dy.\right\}
    \end{eqnarray*}
   Upon substituting \( wy - i= u \) into the last integral, we obtain:
    \begin{eqnarray*}
       &&\int_{\mathbb{R}} \eta(|(T_wg)(y)|)dy\\
        &\leq& \frac{1}{\mu_{0,\pi}(\chi)}\displaystyle\sum_{i\in \mathbb{Z}}\left\{\int_0^1 \eta\left(\mu_{0,\pi}(\chi)\left|g \left( \frac{i+ t^\alpha}{w + 1}\right)\right| \right)dt\right\}||\chi||_1\\
        &\leq&\frac{1}{\mu_{0,\pi}(\chi)}||\chi||_1 I^{\eta}[\mu_{0,\pi}(\chi) g].
        \end{eqnarray*}
\end{proof}
\begin{thm}\label{thm3}
    For every \( g \in L^{\eta}(\mathbb{R}) \), it is possible to find a constant \( \lambda > 0 \) such that the following condition is satisfied:
    \[
    \lim_{w\rightarrow \infty}I^{\eta}[\lambda(T_wg-g)]=0.
    \]
\end{thm}
\begin{proof}
  Let \( g \in L^{\eta}(\mathbb{R}) \).There exists a constant \( \bar{\lambda} > 0 \) and a function \( h \in C_c(\mathbb{R}) \) such that the inequality \( I^{\eta}[\bar{\lambda}(g - h)] < \epsilon \) is true. Next, select a value \( \lambda > 0 \) satisfying the condition \( 3\lambda \left( 1 + \mu_{0, \pi}(\chi) \right) \leq \bar{\lambda} \). By exploiting the properties of \( \eta \) and invoking Theorem \ref{thm2}, one can deduce that:

   \begin{eqnarray*}
       &&I^{\eta}[\lambda(T_wg-g)]=0\\
       &\leq& I^{\eta}[3\lambda(T_wg-T_wh)]+I^{\eta}[3\lambda(T_wh-h)]+I^{\eta}[3\lambda(g-h)]\\
       &\leq& \frac{1}{\mu_{0,\pi}(\chi)}||\chi||_1 I^{\eta}[\bar{\lambda}(g-h)]+I^{\eta}[3\lambda(T_wh-h)]+I^{\eta}[\bar{\lambda}(g-h)]\\
       &\leq& \left(\frac{1}{\mu_{0,\pi}(\chi)}||\chi||_1+1\right)\epsilon+I^{\eta}[3\lambda(T_wh-h)].
       \end{eqnarray*}
       The conclusion is now an immediate consequence of Theorem \ref{thma}.
\end{proof}
\section{Applications to specific classes of Function Spaces}
We utilize the results derived in the notable case where \( \eta(u) = u^p \) for \( u \in \mathbb{R}_0^+ \). In this situation, we observe that \( L^{\eta}(\mathbb{R}) = E^{\eta} = L^p(\mathbb{R}) \) for \( 1 \leq p < \infty \). Within this framework, the equivalence of modular convergence and standard norm convergence holds. Based on the theory outlined in the preceding sections, we now introduce the following corollaries.
\begin{corollary}\label{col1}
    For each \( g \in L^p(\mathbb{R}) \), where \( 1 \leq p < \infty \), the following holds:
\[
    ||T_wg||_p\leq (\mu_{0,\pi}(\chi))^{(p-1)/p}||\chi|_1^{1/p}||g|_p
    \]
    \end{corollary}
    \begin{proof}
      By directly applying Theorem \ref{thm2} with \( \eta(u) = u^p \), we obtain the following inequality:
\[
\| T_w g \|_p^p \leq (\mu_{0, \pi}(\chi))^{p-1} \| \chi \|_1 \| g \|_p^p,
\]
thus the desired conclusion follows immediately.
\end{proof}
\begin{corollary}\label{col2}
    Consider \( g \in L^p(\mathbb{R}) \) with \( 1 \leq p < \infty \). Then, the following assertion is valid:
    \[
    \lim_{w\rightarrow \infty}||T_wg-g||_p=0.
    \]
     \end{corollary}
    In particular, we analyze the function \( \eta_{\alpha, \beta}(v) = v^{\alpha} \log^{\beta}(e + v) \), where \( v \geq 0 \), \( \alpha \geq 1 \), and \( \beta > 0 \). The associated  Orlicz space is then defined as the set of functions \( f \in M(\mathbb{R}) \) for which
\[
I^{\eta_{\alpha,\beta}}[\lambda g] = \int_{\mathbb{R}} \left( \lambda |g(y)| \right)^{\alpha} \log^{\beta}(e + \lambda |g(y)|) \, dy < \infty
\]
for some \( \lambda > 0 \), and is denoted by \( L^{\alpha} \log^{\beta} L(\mathbb{R}) \). 

It is important to note that the function \( \eta_{\alpha,\beta} \) satisfies the \( \Delta_2 \)-condition (see (\ref{cond1})), which implies that \( L^{\alpha} \log^{\beta} L(\mathbb{R}) \) is identical to its finite element space \( E^{\eta_{\alpha,\beta}}(\mathbb{R}) \). Specifically, if \( g \in L^{\alpha} \log^{\beta} L(\mathbb{R}) \), then for every \( \lambda > 0 \), we have \( I^{\eta_{\alpha,\beta}}[\lambda g] < \infty \).

From the general theory, we derive the following results for the specific case where \( \alpha = \beta = 1 \).
\begin{corollary}\label{col3}
     For every $g\in L\log L$, we obtain the following:
     \begin{eqnarray*}
   && \int_{\mathbb{R}}|T_wg(y)|\log(e+\lambda|T_wg(y))|dy\\
   && \qquad \quad\quad \quad \quad \quad \quad \quad \quad \quad \quad \quad  \leq ||\chi||_1\int_{\mathbb{R}}|g(y)|\log\left(e+\lambda \mu_{0,\pi}(\chi)||g(y)|\right)dy\,\,\,\,\,\,(\lambda>0).
\end{eqnarray*}
$\text{Specifically, the operator } T_w f \text{ is well-defined, and if } g  \text{ is an element of }  L \log L, \text{then } T_w g \in L \log L.
$
\newline
Given that for Orlicz spaces fulfilling the \( \Delta_2 \) condition, norm convergence and modular convergence are equivalent, it directly follows that:

   \end{corollary}  
   \begin{corollary}\label{col4}
       For $g\in L\log L$ and for every $\lambda>0$ we have
       \[
       \int_{\mathbb{R}}|T_wg(y)-g(y)|\log(e+\lambda|T_wg(y))-g(y)|dy=0
       \]
       or equivalently,
       \[
       \lim_{w\rightarrow \infty}||T_wg-g||_{L\log L}=0.
       \]
       where \( L \log L \) refers the Luxemburg norm linked to the functional \( I^{\eta_{1,1}} \).
   \end{corollary}
A third case of interest involves the exponential space associated with the \( \eta \)-function defined by \( \eta_{\alpha}(v) = \exp(v^{\alpha}) - 1 \), where \( v\geq 0 \) and \( \alpha > 0 \). In this case, the Orlicz space \( L^{\eta_{\alpha}}(\mathbb{R}) \) includes functions \( g \in M(\mathbb{R}) \) for which the following condition holds:
\[
I^{\eta_{\alpha}}[\lambda g] = \int_{\mathbb{R}} \left( \exp\left( \left(\lambda |g(y)|\right)^{\alpha} \right) - 1 \right) \, dy < \infty
\]

For some \( \lambda > 0 \), since the function \( \eta_{\alpha} \) fails to satisfy the \( \Delta_2 \)-condition (see \ref{cond1}), it follows that \( L^{\eta_{\alpha}}(\mathbb{R}) \) is not equivalent to its finite element space, \( E^{\eta_{\alpha}}(\mathbb{R}) \). Furthermore, convergence in the modular sense in this space does not necessarily imply convergence in norm.
\begin{corollary}\label{col5}
    For each \( g \in L^{\eta_{\alpha}}(\mathbb{R}) \), the following inequality is satisfied:
\begin{eqnarray*}
    &&\int_{\mathbb{R}} \left( \exp\left\{ \left(\lambda |T_w g(y)|\right)^{\alpha} \right\} - 1 \right) \, dy \\
    &&\qquad \leq \frac{\| \chi \|_1}{\mu_{0,\pi}(\chi)} \int_{\mathbb{R}} \left( \exp\left\{ \left( \lambda \mu_{0,\pi}(\chi) |g(y)| \right)^{\alpha} \right\} - 1 \right) \, dy \quad (\lambda > 0).
\end{eqnarray*}
$\text{In particular, } T_w g \text{ is well-defined, and for } g \in L^{\eta_{\alpha}}(\mathbb{R}), \text{ it follows that } T_w g \in L^{\eta_{\alpha}}(\mathbb{R})$.

\end{corollary}
\text{The following corollary is an immediate consequence of Theorem
\ref{thm3}.}
\begin{corollary}\label{col6}
   $\text{For each } g \in L^{\eta_{\alpha}}(\mathbb{R}), \text{ there exists a positive constant } \lambda > 0 \text{ such that}$
    \[\int_{\mathbb{R}} \left( \exp\left\{ \left(\lambda |T_w g(y)-g(y)|\right)^{\alpha} \right\} - 1 \right) \, dy=0.\]
\end{corollary}
\section{Applications and examples with graphical representations}
Next, we analyze the sampling series \( T_w(x) \), which is constructed using following Fejér's kernel:
\[
F(y)=\frac{1}{2}sinc^2\left(\frac{y}{2}\right),\,\,\,\,\,\,\,(y\in \mathbb{R})
\]
where the sinc function and its Fourier transform, interpreted in the context of the \( L^2 \)-space, are defined as follows:
\[
\text{sinc}(y) :=
\left\{
\begin{array}{ll}
   \frac{\sin(\pi y)}{\pi y}, & y \in \mathbb{R} \setminus \{0\}, \\
   1, & y = 0.
\end{array}
\right.
\quad
\mathcal{F}(\text{sinc}(y)) =
\text{rect}(\mu) :=
\left\{
\begin{array}{ll}
   1, & |\mu| \leq \pi, \\
   0, & |\mu| > \pi.
\end{array}
\right.
\]
\begin{figure}[h] % "h" means here; image is inserted at the point in the document where the code is placed
    \centering
    \includegraphics[width=0.7\textwidth]{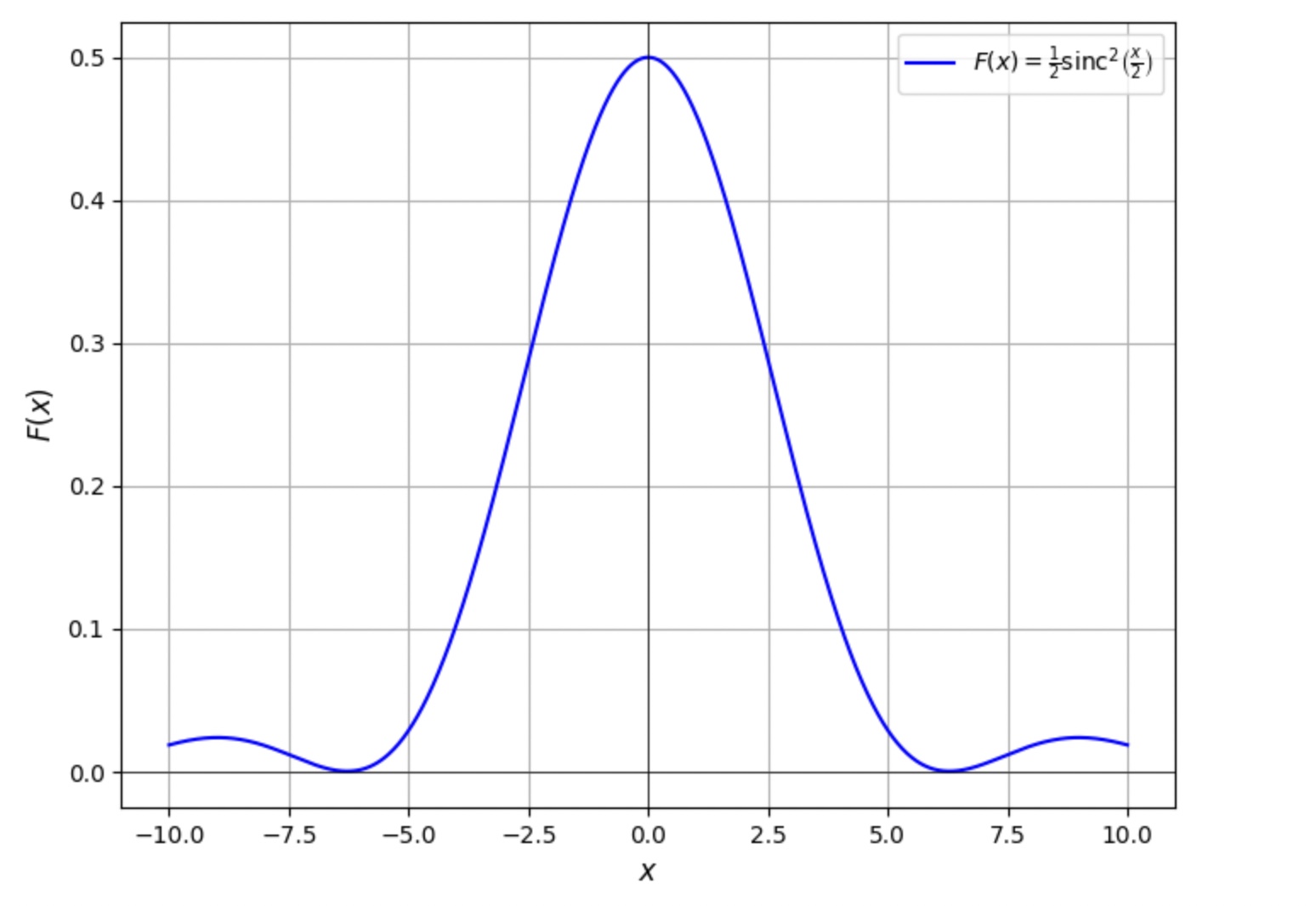} 
    \caption{Graph of \( F(y) = \frac{1}{2} \text{sinc}^2\left(\frac{y}{2}\right) \)} % Optional caption
    \label{fig:screenshot} % Optional label for referencing
\end{figure}
Clearly, the function \( F \) belongs to \( L^1(\mathbb{R}) \), and fulfills the asymptotic condition
\[
F(y) = \frac{1}{2} \, \text{sinc}^2\left( \frac{y}{2} \right) = O\left( \frac{1}{y^2} \right) \quad \text{as} \quad y \to \infty.
\]
Consequently, there exist a constant \( M \) such that whenever \( |i - u| > N \), the following inequality holds:
\[
|F(z - i)| |z - m|^\beta \leq \frac{M}{|z - i|^2} |z - i|^\beta = \frac{M}{|z - i|^{2 - \beta}}.
\]
Additionally, since \( F(x) \) is smooth and continuous, there exists a constant \( C_1 \) (depending on \( N \)) such that for \( |z - i| \leq N \), we have
\[
|F(z -i)| \leq C_1 \quad \text{for all} \quad |z - i| \leq N.
\]

Hence, we can bound the sum as follows:
\[
\sum_{i\in \mathbb{Z}} |F(z - i)| |z - i|^\beta = \sum_{|i- u| \leq N} |F(z - i)| |z - i|^\beta + \sum_{|i - z| > N} |F(z - i)| |z - i|^\beta.
\]
For the first term, since the sum involves a finite number of terms, it is clearly finite. For the second term, we use the estimate derived earlier:
\[
\sum_{|i- z| > N} |F(z - i)| |z - i|^\beta \leq \sum_{|i- z| > N} \frac{M}{|z - i|^{2 - \beta}}.
\]
This second sum converges provided that \( 2 - \beta > 1 \), or equivalently, \( \beta < 1 \), because the terms decay like \( \frac{1}{|i|^{2 - \beta}} \), and the series \( \sum_{i = N}^\infty \frac{1}{i^{2 - \beta}} \) converges for \( \beta < 1 \).

Thus, \( F \) satisfies the moment condition (\ref{eql2}) with \( \beta < 1 \).

Then, by applying the Poisson summation formula, we obtain the following result:
\[
\sum_{i \in \mathbb{Z}} F(z - i) = \sum_{i \in \mathbb{Z}} \mathcal{F}(F(2\pi i)),
\]
here \( \mathcal{F}(f) \) refers to the Fourier transform of the function \( f \), and 
\[
\mathcal{F}(F(y)) = 
\begin{cases} 
1 - \left| \frac{y}{\pi} \right|, & \text{if} \quad |y| \leq \pi, \\
0, & \text{if} \quad |y| > \pi.
\end{cases}
\]
\vspace{1em}
From the equations above, it is evident that \( F \) satisfies the moment condition (\ref{eql1}).\\
The Kantorovich-type sampling series \ref{eq2} of $f\in L^p(\mathbb{R}), 1\leq p<\infty$ now becomes:
\begin{eqnarray}
   (T_w^F f)(y) = \sum_{i \in \mathbb{Z}} \left[ \int_0^1 f \left( \frac{i+ t^\alpha}{w + 1} \right) \, dt \right]\frac{1}{2}sinc^2\left(\frac{wy-i}{2}\right), \quad \text{for} \, y \in \mathbb{R}. 
\end{eqnarray}
The results of Corollaries \ref{col2}, \ref{col4}, and \ref{col6} are applicable to \( T_w^F \).
\vspace{1em}
Next, we apply the sampling operator \( T_w^F \) to two separate discontinuous functions, namely:
\[
f_1(y) :=
\left\{
\begin{array}{ll}
1, & |y|\leq 1, \\
   0, & otherwise.
\end{array}
\right.
\]
and 
the function \( f_2(x) \) is defined as:
\[
f_2(y) :=
\begin{cases}
\quad\frac{9}{y^2}, &\quad x < -3, \\[5pt]
\quad2, & -3 \leq y < -2, \\[5pt]
-\frac{1}{2}, & -2 \leq y < -1, \\[5pt]
\quad\frac{3}{2}, & -1 \leq y < 0, \\[5pt]
\quad1, & \quad0 \leq y < 1, \\[5pt]
-1, & \quad1 \leq y < 2, \\[5pt]
\quad0, & \quad 2 \leq y < 3, \\[5pt]
-\frac{50}{y^4}, & \quad3 \leq y.
\end{cases}
\]
\clearpage
\begin{figure}[h]
    \centering
    \includegraphics[width=0.7\textwidth]{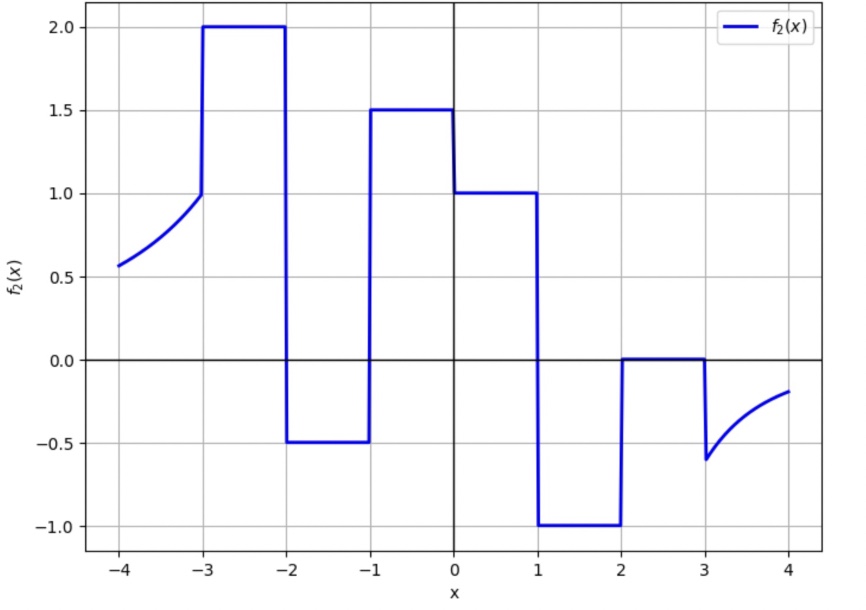} 
    \caption{Graph of the function $f_2$} 
    \label{fig:screenshot} 
\end{figure}
The Kantorovich-type sampling series $(\ref{eq2})$ using Fejér's Kernel (\( F \)) approximates functions by integrating over an infinite range, but the kernel’s infinite support complicates the process, requiring an infinite number of terms. For functions with \textbf{compact support} (like \( f_1 \)), only a finite number of terms contribute significantly, making the series tractable. However, for functions with \textbf{infinite support} (like \( f_2 \)), an infinite number of terms are relevant, and the series must be truncated, introducing \textbf{truncation error}.\\
The approximation to function $f_1$ and $f_2$ by the operators $S_w^F$ and $T_w^F$ are shown in the figures $\ref{fig4}$, $\ref{fig5}$, $\ref{fig6}$ and $\ref{fig7}$.  We can see that $T_w^F$ presents an order of approximation better than the operator $S_w^F$.\\
Also figures \(\ref{fig5}\) and \(\ref{fig7}\) illustrate how truncation error decreases as the sampling parameter \( w \) increases, but a perfect approximation is unattainable without considering an infinite number of terms. Truncation introduces a balance between computational efficiency and accuracy, with a larger \( w \) improving the approximation at the cost of increased computation.
\vfill
\begin{figure}[h]
    \centering
    \includegraphics[width=0.7\textwidth]{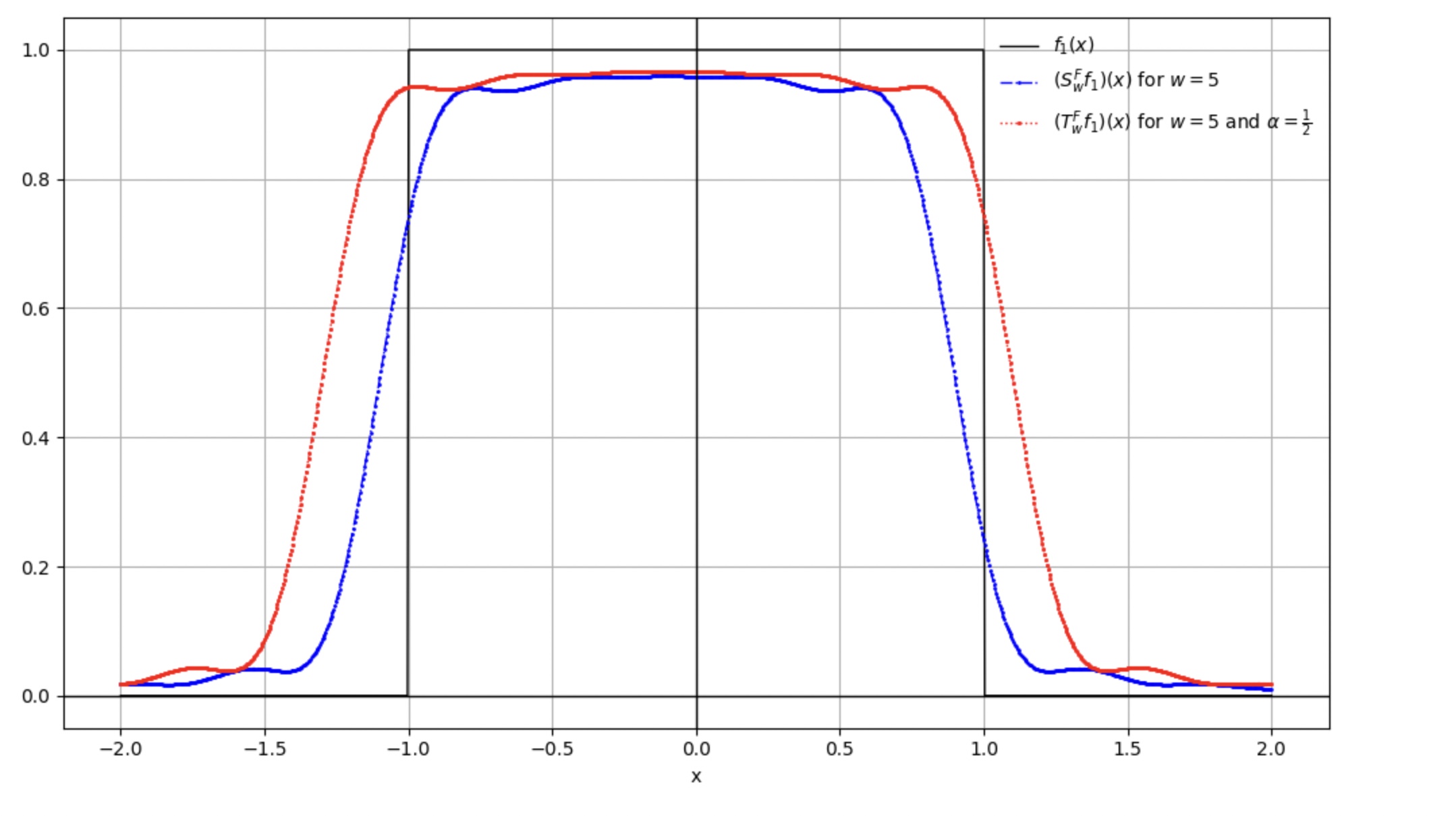} 
    \caption{Comparison of $f_1(x)$, $(S_w^F f_1)(x)$ and $(T_w^F f_1)(x)$ for $w = 5$ and $\alpha = 1/2$} 
    \label{fig4} 
\end{figure}
\begin{figure}[h]
    \centering
    \includegraphics[width=0.7\textwidth]{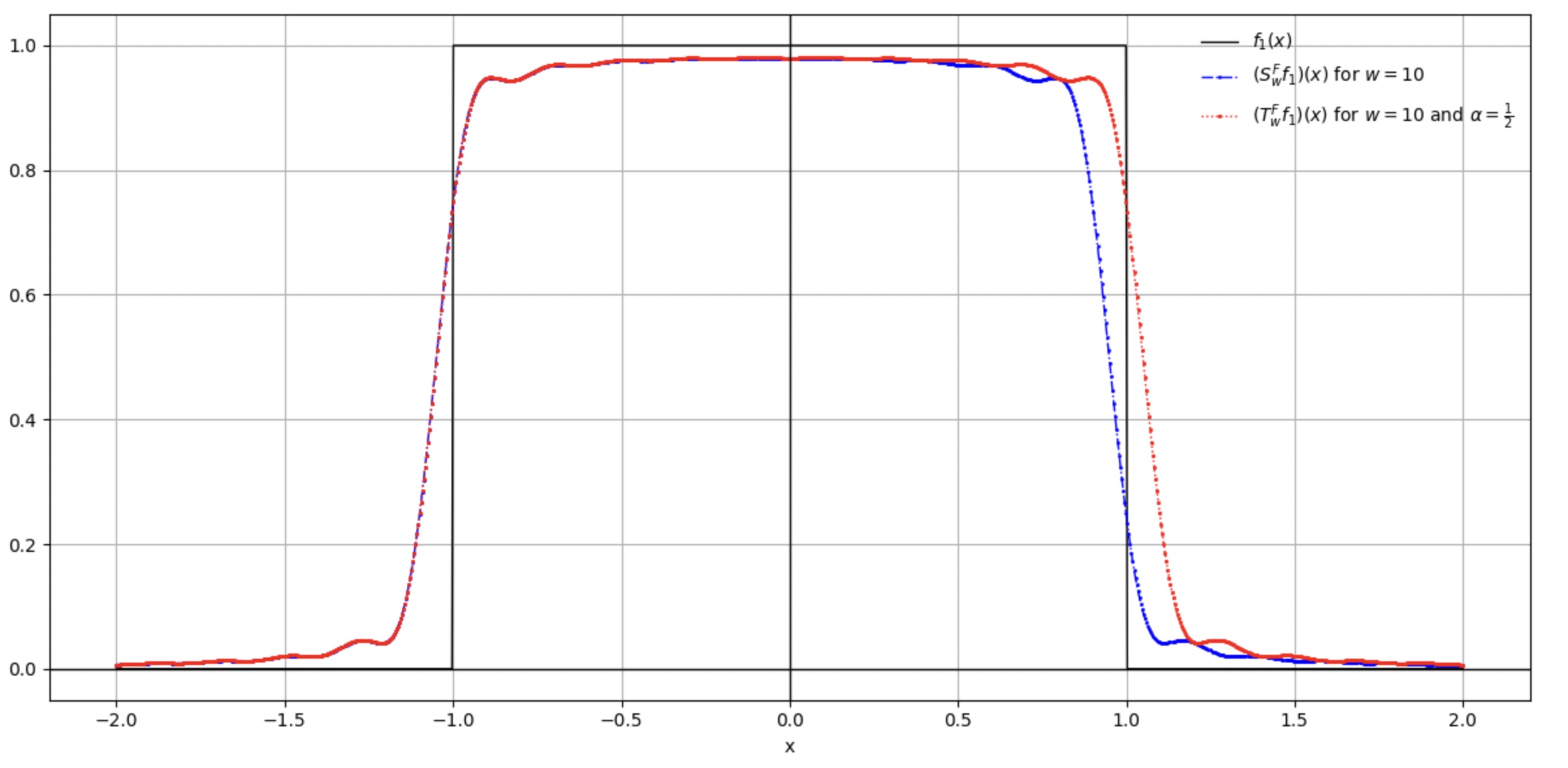} 
    \caption{Comparison of $f_1(x)$, $(S_w^Ff_1)
    (x)$ and $(T_w^F f_1)(x)$ for $w = 10$ and $\alpha = 1/2$} 
    \label{fig5} 
\end{figure}
\begin{figure}[h]
    \centering
    \includegraphics[width=0.7\textwidth]{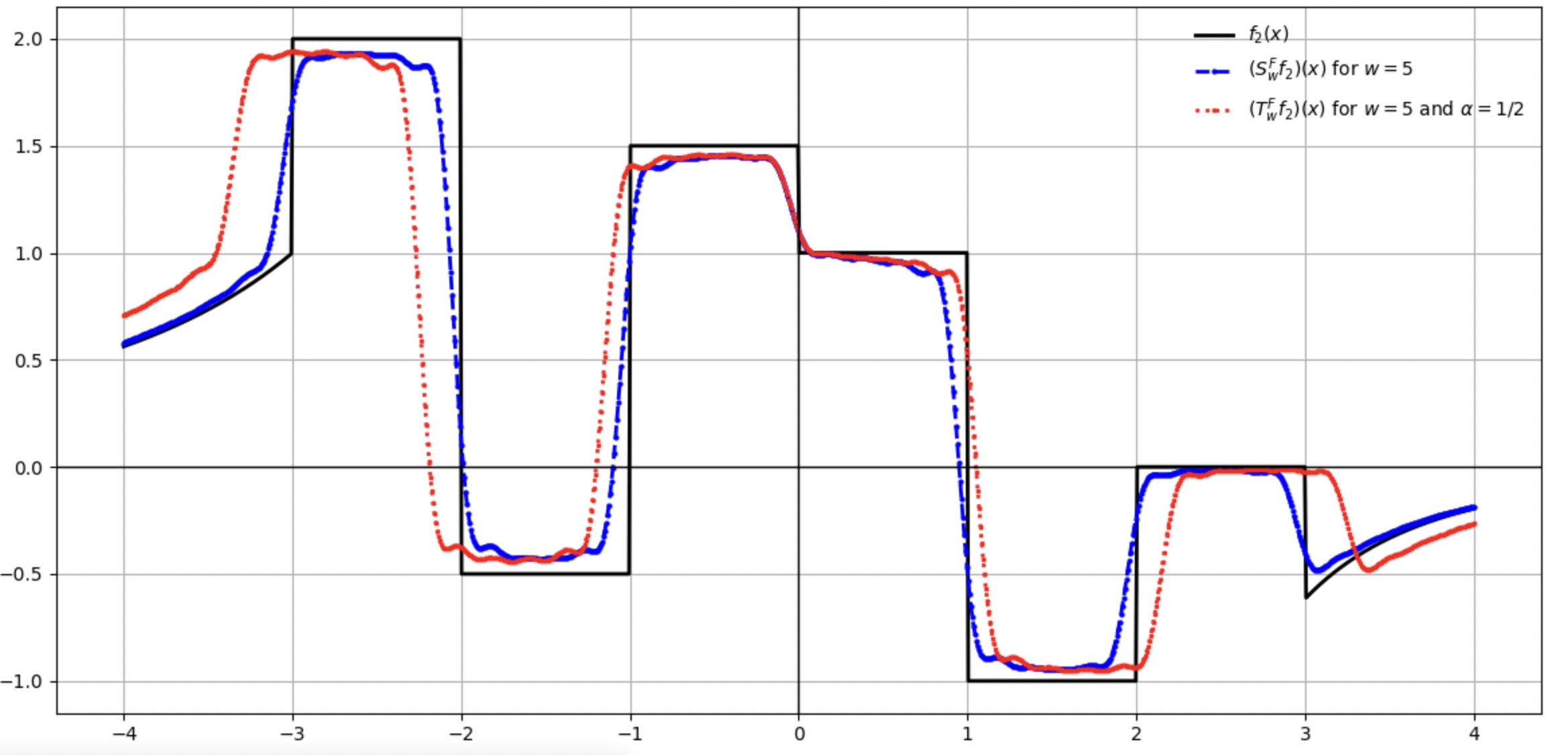} 
    \caption{Comparison of $f_2(x)$, $(S_w^Ff_2)
    (x)$ and $(T_w^F f_2)(x)$ for $w =5$ and $\alpha = 1/2$} 
    \label{fig6} 
\end{figure}
\begin{figure}[h]
    \centering
    \includegraphics[width=0.7\textwidth]{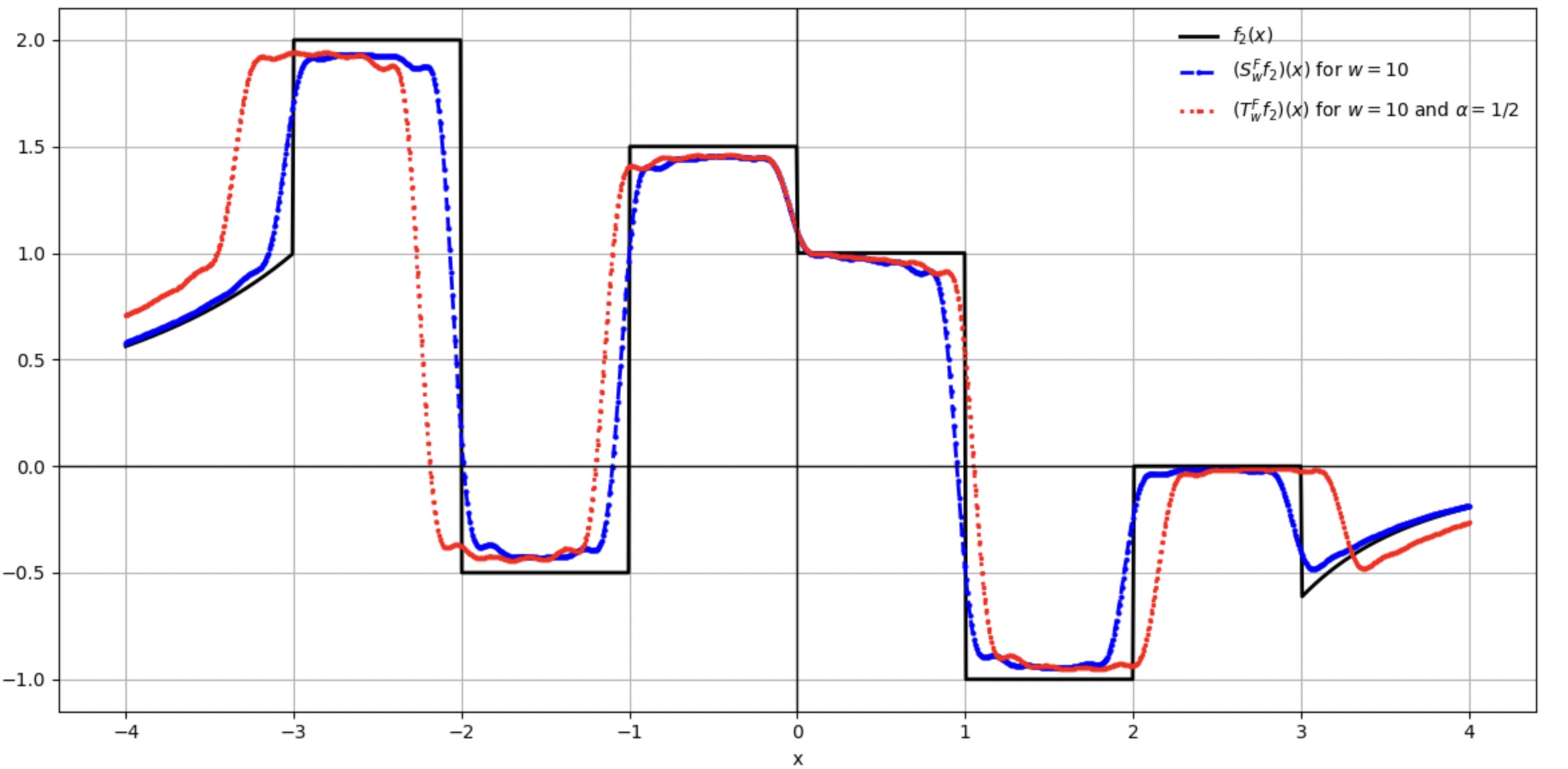} 
    \caption{Comparison of $f_2(x)$, $(S_w^Ff_2)
    (x)$ and $(T_w^Ff_2)(x)$ for $w =10$ and $\alpha = 1/2$} 
    \label{fig7} 
\end{figure}
To minimize the truncation error, kernels \( \chi(x) \) can be selected that decrease more rapidly as \( x \to \pm \infty \). A family of such kernels is the class of generalized Jackson kernels, defined as:
\[
J_n(x) = c_n \, \text{sinc}^{2n} \left( \frac{x}{2n\pi \alpha} \right), \quad (x \in \mathbb{R}), \quad c_n := \left( \int_{-\infty}^{\infty} \text{sinc}^{2n} \left( \frac{u}{2n\pi \alpha} \right) \, du \right)^{-1},
\]

where \(n \in \mathbb{N}\) and \(\alpha \geq 1\). These kernels are bandlimited to the frequency range \( \left[ -\frac{1}{\alpha}, \frac{1}{\alpha} \right] \), implying that their Fourier transform vanishes outside this range. Consequently, the conditions in equations $(\ref{eql1})$ and $(\ref{eql2})$ are satisfied , as noted in Remark \ref{re1} \ref{itb}, \ref{itc}. The results of Corollaries \ref{col1}, \ref{col4}, and \ref{col5} are applicable to \( S_w^F \) where the Fejér kernel \( F \) is substituted by the generalized Jackson kernel \( J_n \).
\\
To completely eliminate the truncation error, it is essential to consider kernels \(\chi\) that possess bounded support. One of the most practical and commonly used examples of such kernels are the \textit{B-splines} of order \(n \in \mathbb{N}\), which are defined as:
\[
M_n(y)=\frac{1}{(n-1)!}\sum_{m=0}^n (-1)^m \binom{n}{m}\left(\frac{n}{2}+y-m\right)^{n-1}_+
\]
with Fourier transform given by
\[
\mathcal{M}_n(u)=sinc^n\left(\frac{u}{2\pi}\right)\,\,\,(u\in \mathbb{R}).
\]
For every \( n \in \mathbb{N} \), the functions \( M_n \) are bounded on \( \mathbb{R} \) and have compact support contained within the interval \( \left[ -\frac{n}{2}, \frac{n}{2} \right] \).
This implies that \( M_n \in L^1(\mathbb{R}) \), and the moment condition \((\ref{eql2})\) holds for all \( \beta > 0 \).
Moreover, the normalization condition \((\ref{eql1})\) is satisfied, as detailed in Remark \ref{re1} \ref{itb}.
 Consequently, the results from Corollaries \ref{col1}, \ref{col4}, and \ref{col5} follow for \( S_w^{M_n} \).\\
Let us examine the case \( n = 3 \) in more detail. The B-spline \( M_3 \) is defined as:
\[
M_3(y) :=
\begin{cases}
\frac{3}{4}-y^2, &|y|\leq \frac{1}{2}, \\[5pt]
\frac{1}{2}\left(\frac{3}{2}-|y|\right)^2, & \frac{1}{2} \leq |y|\leq \frac{3}{2}, \\[5pt]
0, & |y|>\frac{3}{2}. 
\end{cases}
\]
\vfill
\begin{figure}[h] 
    \centering
    \includegraphics[width=0.7\textwidth]{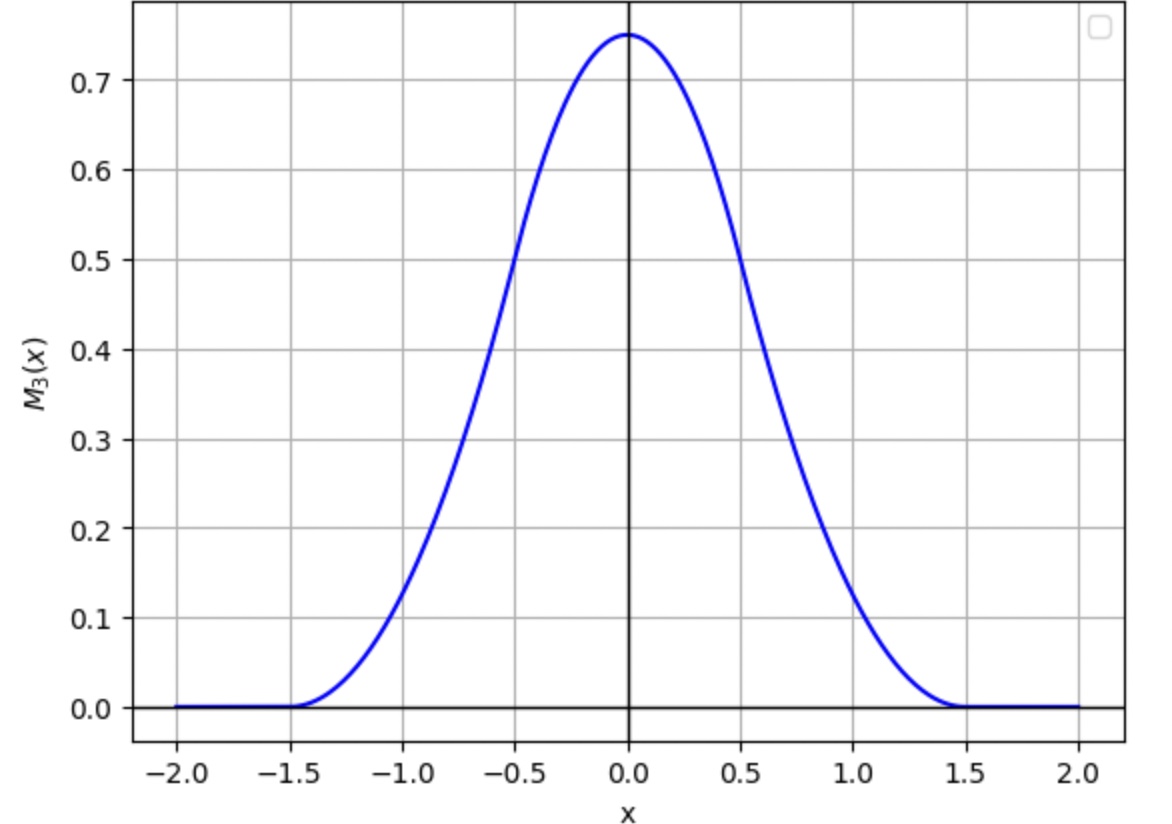} 
    \caption{The Spline Kernel $M_3.$} 
    \label{fig:screenshot} 
\end{figure}
The Kantorovich-type sampling series, $S_w^{M_n}$ and $T_w^{M_n}$ with parameters \( w = 5 \) and \( w = 10 \), applied to the functions \( f_1 \) and \( f_2 \), produce noteworthy results. The plots $\ref{fig8}$, $\ref{fig9}$, $\ref{fig10}$ and $\ref{fig11}$ indicate that the approximation achieved using the series based on the spline kernel \( M_3 \) performs better than the one using the series based on Fejér's kernel, even when comparing the spline series with \( w = 5 \) to the Fejér kernel series for \( w = 10 \). This suggests that the series \( S_w^{M_3} f \) and \( T_w^{M_3} f \) can offer a more accurate approximation with fewer mean values of \( f \) than the series \( S_w^F \) and \( T_w^F \). Also it is noteworthy to note that $T_w^{M_3}$ presents an order of approximation better than the operator $S_w^{M_3}$.
\clearpage
\begin{figure}[h]
    \centering
    \includegraphics[width=0.9\textwidth]{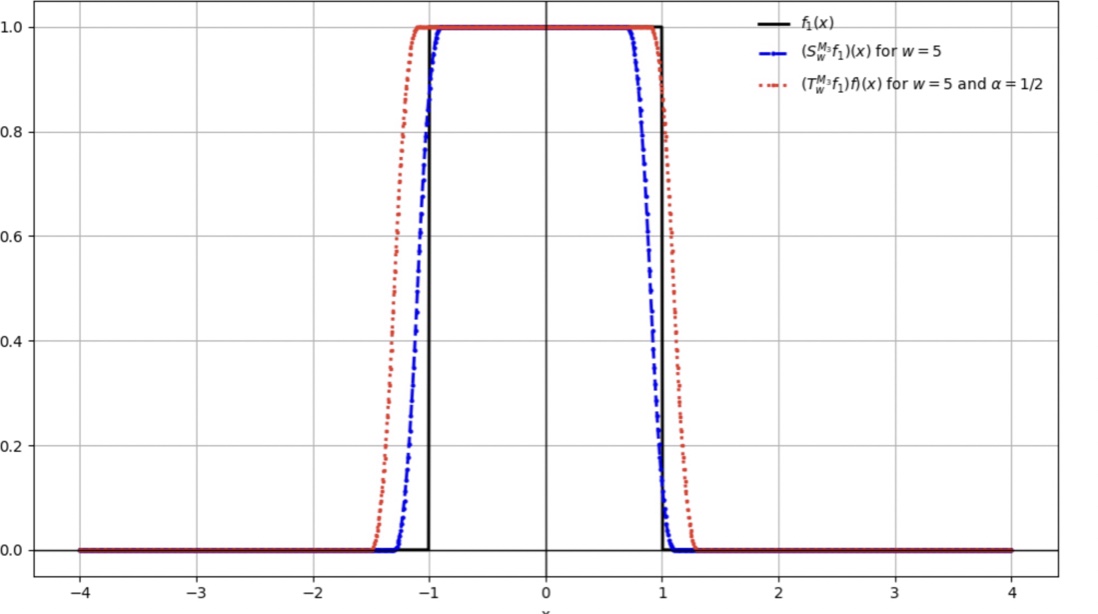} 
    \caption{Comparison of $f_1(x)$, $(S_w^{M_3}f_1)(x)$ and $(T_w^{M_3} f_1)(x)$ for $w = 5$ and $\alpha = 1/2$} 
    \label{fig8} 
\end{figure}
\begin{figure}[h]
    \centering
    \includegraphics[width=0.9\textwidth]{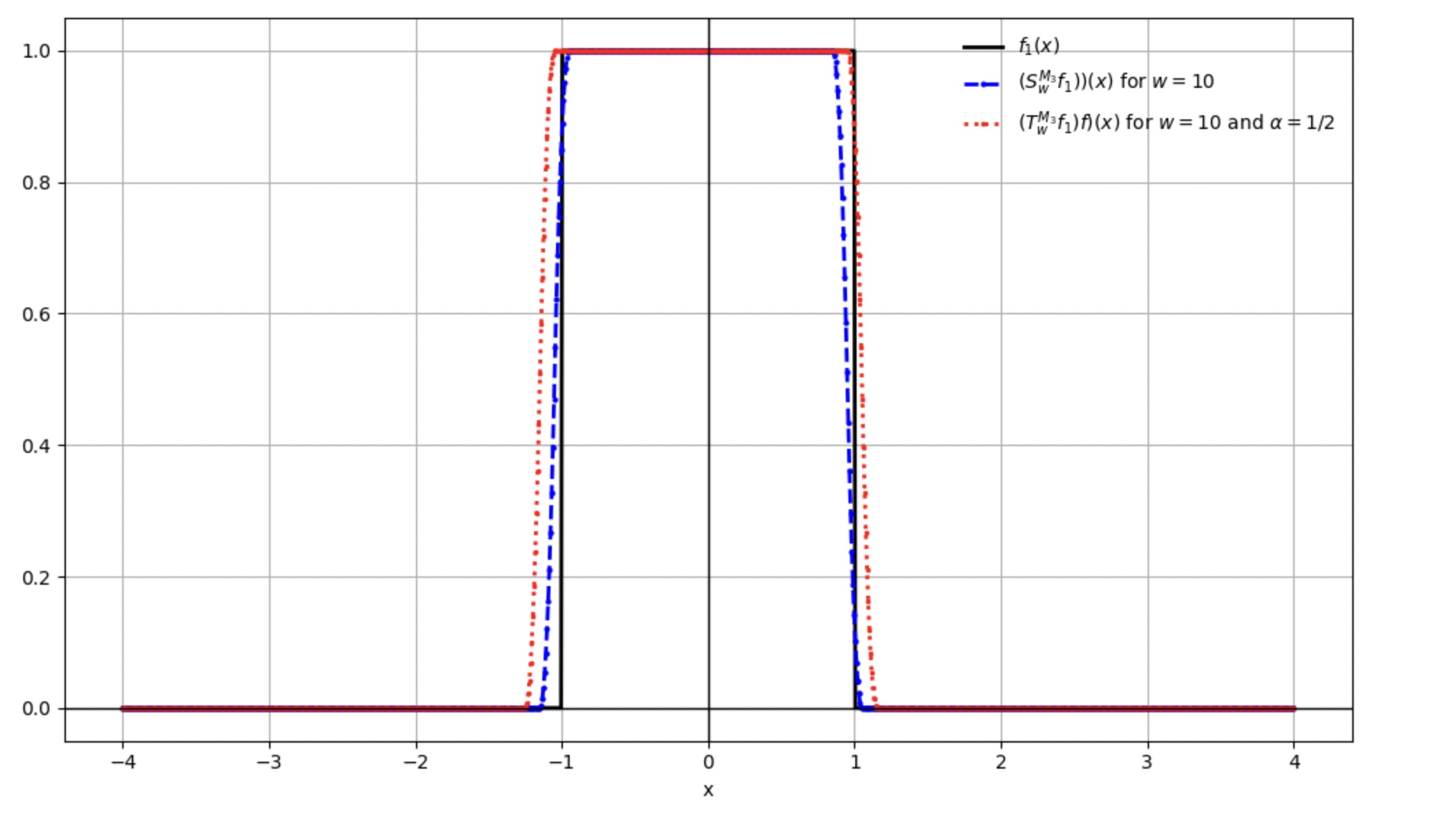} 
    \caption{Comparison of $f_1(x)$, $(S_w^{M_3}f_1)(x)$ and $(T_w^{M_3} f_1)(x)$ for $w = 10$ and $\alpha = 1/2$} 
    \label{fig9} 
\end{figure}
\begin{figure}[h]
    \centering
    \includegraphics[width=0.9\textwidth]{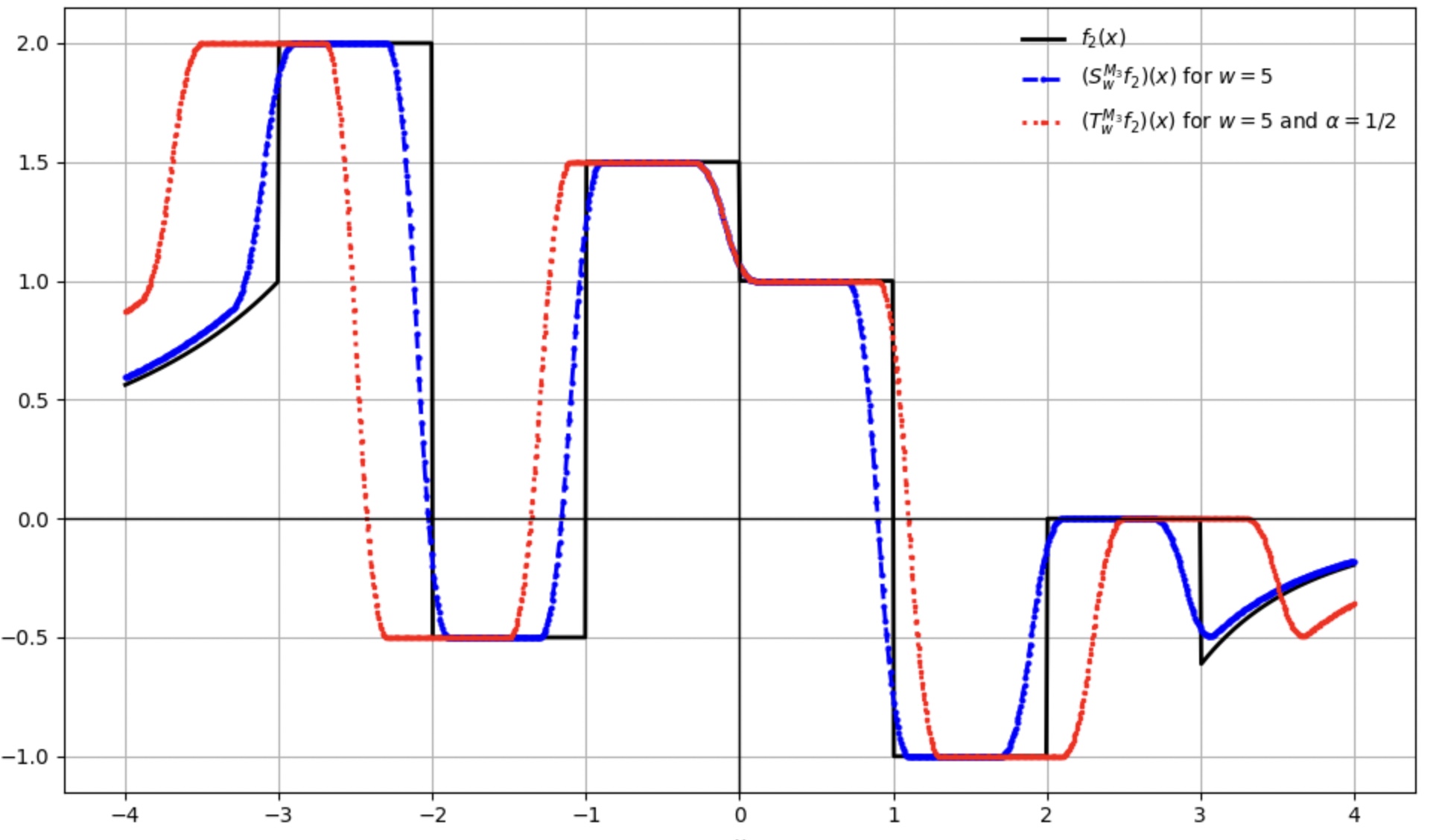} 
    \caption{Comparison of $f_2(x)$, $(S_w^{M_3}f_2)(x)$ and $(T_w^{M_3} f_2)(x)$ for $w =5$ and $\alpha = 1/2$} 
    \label{fig10} 
\end{figure}
\begin{figure}[h]
    \centering
    \includegraphics[width=0.9\textwidth]{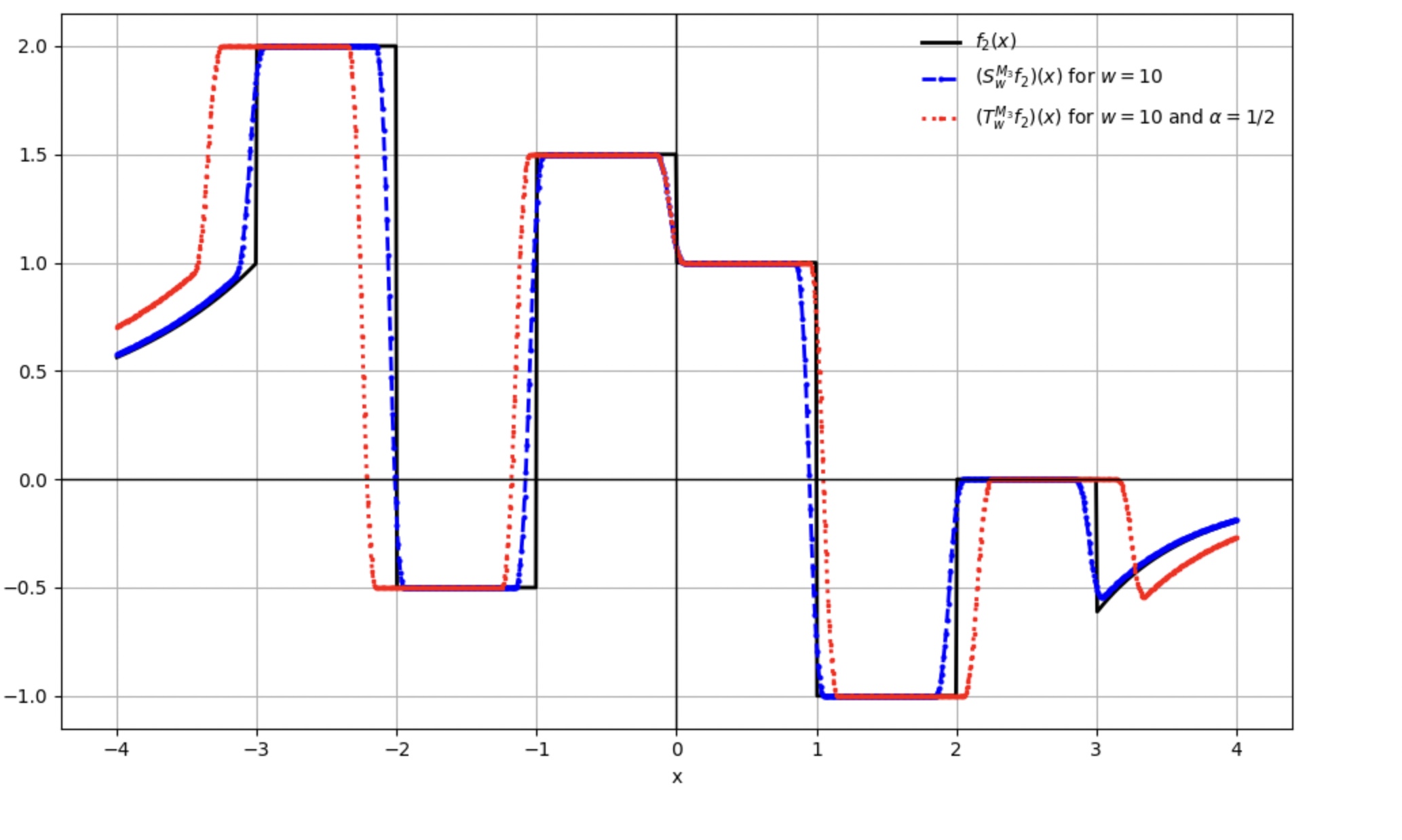} 
    \caption{Comparison of $f_2(x)$, $(S_w^{M_3}f_2)(x)$ and $(T_w^{M_3} f_2)(x)$ for $w =10$ and $\alpha = 1/2$} 
    \label{fig11} 
\end{figure}

\end{document}